\documentclass{article}
\usepackage[utf8]{inputenc}
\usepackage{titlesec}

\usepackage{amsthm}
\usepackage{amsmath}
\usepackage{amsfonts}
\usepackage{amssymb}
\usepackage{enumerate}
\usepackage{tikz-cd}
\usepackage{mathtools}
\usepackage{cleveref}
\usepackage[margin=1.5in]{geometry}
\usepackage{wrapfig}
\tikzcdset{arrow style=math font}

\titleformat{\section}
  {\normalfont\fontsize{11}{12}\bfseries}{\thesection}{1em}{}

\titleformat{\subsection}[runin]
  {\normalfont\fontsize{10}{11}\bfseries}{\thesubsection}{.5em}{}
  
\theoremstyle{plain}
  \newtheorem{thm}{Theorem}[section]
  \newtheorem{cor}[thm]{Corollary}
  \newtheorem{lemma}[thm]{Lemma}
\newtheorem{prop}[thm]{Proposition}
\newtheorem{Ex}[thm]{Example}

\theoremstyle{definition}
\newtheorem{dfn}[thm]{Definition}
\newtheorem{rmk}[thm]{Remark}

\numberwithin{equation}{section}

\title{O'Grady's Birational Maps via Wall-hitting}
\author{Huachen Chen}
\date{}

\begin{document}
\maketitle

\begin{abstract}

 \noindent We observe that O'Grady's birational maps \cite{O'G1} between moduli of sheaves on an elliptic K3 surface can be interpreted as intermediate wall-crossing (wall-hitting) transformations at the so-called totally semistable walls, studied by Bayer and Macr\`i \cite{BM2}. As an ingredient to prove this observation, we describe the first totally semistable wall for ideal sheaves of $n$ points on the elliptic K3. We then use this observation to make a remark on Marian and Oprea's strange duality \cite{MO}.    
\end{abstract}

\section{Introduction}
In \cite{O'G1}, K. O'Grady constructed, in particular, a series of birational maps $$\psi_r: M_H (v_r) \dashrightarrow M_H(v_{r+1}),$$ where $M_H(v_r)$ is the moduli of Gieseker $H-$semistable rank $r$ sheaves on an elliptic K3 surface and of class $v_r \in H^*(X, \mathbb Z).$ In \cite{BM2}, A. Bayer and E. Macr\`i described the wall-crossing behavior of moduli of complexes on a K3 surface $X$ with fixed class $v \in H^*(X, \mathbb Z)$, which in particular revealed the birational geometry of moduli of sheaves on $X$. \textit{The main goal of this note} is to show that O'Grady's birational maps appear as intermediate wall-crossing (also refer to as wall-hitting) transformations at totally semistable walls.

Our motivation comes from the study of strange duality for K3 surfaces: in \cite{MO}, A. Marian and D. Oprea interpret O'Grady's birational maps, from $M_H(v_r)$ to Hilbert scheme of points $X^{[a]},$ as a Fourier-Mukai transform and use that to propagate strange duality isomorphisms on Hilbert schemes of points to a large class of pairs of moduli spaces of higher rank sheaves; on the other hand, work of Bayer and Macr\`i, \cite{BM1} and \cite{BM2}, enable one to extend strange duality to Moduli of complexes setting and to study them by wall-crossing. Their work indicates that crossing totally semistable walls is particularly interesting for strange duality. As an application of our main theorem, we obtain new examples of strange duality for elliptic K3 surfaces, based on a result of Marian and Oprea.  

\subsection{O'Grady's birational maps as wall-crossing transformations.}\label{sec1.1}
We recall the construction in \cite{O'G1}, see also \cite{MO}. Assume that $\pi: X \rightarrow \mathbb P^1$ is an elliptic K3 surface whose N\'evon-Severi group is $NS(X)=\mathbb Zc \oplus \mathbb Zf,$ where $c$ and $f$ are the classes of a section and a fiber of $\pi,$ respectively. Therefore, we have $c^2=-2,\  c\cdot f=1,\  f^2=0.$ 

Now fix a Mukai vector $v_r=(r, c+kf, p)\in H^0(X, \mathbb Z)\oplus H^{1,1}(X, \mathbb Z)\oplus H^4(X, \mathbb Z),$ choose a polarization $H:= c+mf,$ where $m$ is sufficiently large. Suppose that $F_r$ is a $H-$stable torsion-free sheaf with class $v(F_r)=v_r.$ Twisting by $\mathcal O(f)$ if necessary, we may assume $\chi(F_r)=-1.$ It can be shown that $ Ext^1( F_r, \mathcal O_X) \cong \mathbb C$ for $F_r \in M_H(v_r)$ general in moduli. The corresponding nontrivial extension 

\begin{equation}\label{extension}
   0 \rightarrow \mathcal O_X \rightarrow F_{r+1} \rightarrow F_r \rightarrow 0 \tag{$\star$}
\end{equation}

\noindent produces a rank $r+1$ sheaf $F_{r+1}$ which is $H$-stable provided that $F_r \in M_H(v_r)$ is away from a particular codimension one locus. Thus we obtain a birational map $$\psi_r: M_H(v_r) \dashrightarrow M_H(v_{r+1}).$$

%This is the map that we claim to appear as a intermediate transformation described by Bayer-Macr\`i in \cite{BM2}.

On the othe hand, in \cite{BM2} Bayer and Macr\`i study systematically a wall-crossing behavior known as totally semistable wall-crossing. (In the case of skyscraper sheaves, this kind of walls had been discovered and exploited by Bridgeland \cite{Bri2}.) At a totally semistable wall $\mathcal W$ for a class $v$, the moduli spaces $M_{\sigma_+}(v)$ and $M_{\sigma_-}(v)$ on the two sides share no common objects, where $\sigma_+$ and $\sigma_-$ are two stability conditions separated by $\mathcal W.$ However, they are linked via an intermediate moduli $M_{\sigma_0}(v_0),$ which parametrizes objects of another class $v_0$ and stable with respect a stability condition $\sigma_0\in \mathcal W.$ More precisely, there are birational maps 
\begin{center}
\begin{tikzcd}[column sep=tiny]
M_{\sigma_+}(v) & & M_{\sigma_-}(v) \\
       & M_{\sigma_o}(v_o)  \arrow[ul, dashed, "\phi_+"] \arrow[ru, dashed, "\phi_-"]
\end{tikzcd}. 
\end{center} 
Morevoer, these birational maps $\phi_\pm$ are induced by spherical twists or inverse spherical twists, see \cref{sec2} for detail. We refer to $\phi_\pm$ as \textit{wall-hitting transformations} at the wall $\mathcal W.$

Our observation is that the extension \eqref{extension} inducing O'Grady's maps is precisely the defining triangle of an inverse spherical twist:

\begin{equation*}
ST^{-1}_{\mathcal O_X}(F_r) \rightarrow F_r \rightarrow \bigoplus_i R^{i}Hom(\mathcal O_X, F_r)\otimes \mathcal O_X[2-i] \cong \mathcal O_X[1].
\end{equation*}

\noindent One sees that $F_{r+1} \cong ST^{-1}_{\mathcal O_X}(F_r).$ Besides, up to twisting by $\mathcal O(f),$ we have $$(v(\mathcal O_X), v_{r+1})=-\chi(\mathcal O_X, F_{r+1}) = -1,$$ which satisfies a lattice-theorectic criterion for a totally semistable wall of $v_{r+1}$ to occur, according to Bayer-Macr\`i's classification of walls (see \cref{BM5.7}).

With these coincidences, it is natural to conjecture that O'Grady's birational maps $\psi_r$ are in fact the wall-hitting transformations at some totally semistable walls. We show that this is true. Moreover, one can manage to pass these walls consecutively:

\begin{thm}\label{thm0}
For any integer $r\geq 1,$ there exists a path $\gamma$ in the space of stability conditions $Stab^\dag(X),$ starting from the Gieseker chamber of $v_r,$ and passing through a wall $\mathcal W_i$ for each $v_i,$ $1\leq i \leq r,$ in the decreasing order (see \cref{firstwalls}), such that $\mathcal W_i$ is the first totally semistable of $v_i$ along $\gamma$. Moreover, one of the wall-hitting transformations at $\mathcal W_i, 1\leq i \leq r,$ can be identified with O'Grady's birational maps $\psi_{i-1}^{-1}: M_H(v_i)\dashrightarrow M_H(v_{i-1})$ generically.
\end{thm}

\begin{figure}
    \centering
    \includegraphics[width=0.8\textwidth]{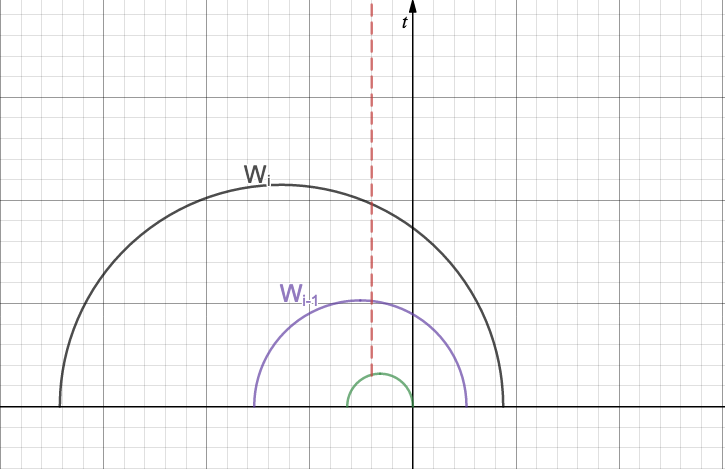}
    \caption{the first totally semistable wall $\mathcal W_i$ of $v_i, 1\leq i \leq r$.}
    \label{firstwalls}
\end{figure}

Again $H:=c+mf$ is an ample class with $m>>0.$ It determines a two dimensional slice (see \cref{firstwalls}) of stability conditions $P_H:=\{\sigma_{uH, tH}: u\in \mathbb R, t>0\}$ (\cref{slice}), on which the walls $\mathcal W_i, 1\leq i\leq r$ are nested semicircles, as we will see in \cref{lemma4}. (Note that these are walls of different classes, so a priori could intersect each other.) The path $\gamma$ can be taken to be a vertical ray (the red dashed line in \cref{firstwalls}) on $P_H,$ starting from $\infty$ and going downward.

Thus, O'Grady's classical result that $M_H(v_r)$ is birational to $X^{[n]},$ the Hilbert scheme of $n$ points on $X,$ can be obtained by composing the wall-hitting transformations at these totally semistable walls that $\gamma$ crosses.

\begin{rmk}
Notice that we identify the wall-hitting transformations with O'Grady's maps $M_H(v_r)\dashrightarrow M_H(v_{r-1})$ only up to birational equivalences, because: first, we can only expect a general element in the moduli of complexes $M_\sigma(v_i),$ where $\sigma$ is above $\mathcal W_i,$ to be a sheaf, but not all; second, on the moduli of complexes side, the spherical twist in fact induces an isomorphism in codimension one, according to \cite{BM2} (see \cref{BMthm2}), while the O'Grady's maps are given directly by the extension process only on the complement of a divisor. The isomorphim in codimension one induced uniformally by a spherical twist will be important for our application to strange duality.
\end{rmk}

As the first step to prove \cref{thm0}, we describe the first totally semistable wall for ideal sheaves of $n$ points on the elliptic K3 surface $X$:

\begin{prop} \label{prop0}
There exists a path $\gamma_1$ in $Stab^\dag(X),$ strating from the Gieseker chamber of ideal sheaves $I_Z$ of $n$ points on $X,$ such that the first totally semistable wall for $I_Z$ along $\gamma_1$ is caused by $\mathcal O(-C)\hookrightarrow I_Z,$ where $C$ is a genus $n$ curve that contains Z.
\end{prop}

\begin{rmk}
The geometric intuition of this proposition is that given $n$ points on a K3 surface, there is always a genus $g$ curve that pass through these points, as long as $g\geq n.$ As in \cite{AB}, the ideal sheaf of this curve $\mathcal O(-C),$ as a subobject of $I_Z$ for all $Z \in X^{[n]},$ can become destabling and so potentially cause a totally semistable wall.  
\end{rmk}

\subsection{Strange duality for elliptic K3 surfaces.}\label{sec1.4} As an application of \cref{thm0}, we give two examples of strange duality isomorphisms for an elliptic K3 (\cref{ex1} and \cref{ex3}), based on a result of Marian and Oprea. Along the way, we also generalize an observation by Bayer and Macr\`i about wall-crossing for strange duality (\cref{prop4}).

We recall very briefly the setup of strange duality for a K3 surface $X$ (see e.g. \cite{MO1} for details). Let $v, w \in H^*_{alg}(X)$ be a pair of Mukai vectors with $(v, w^\vee)=0.$ On the moduli space $M_H(v)$ of stable torsion-free sheaves of class $v,$ one has a determinant line bundle $\theta_v(w),$ depending the orthogonal class $w$ (see \cref{detlb}). Symmetrically, we have another moduli space $M_H(w)$ with a line bundle $\theta_w(v).$ It has been observed that for some choices of $v$ and $w$, we have the following so-called strange duality phenomena:
\begin{enumerate}[(1)]
\item $h^0(M_H(v), \theta_v(w)) = h^0(M_H(w), \theta_w(v));$

\item Moreover, sometimes there exists a geometric explanation of the above equality: \textit{under certain assumptions}, the locus $$\Theta:=\{(E,F)\in M_H(v)\times M_H(w)\ |\ H^0(X, E\otimes F)\cong Ext^1(F^\vee[1], E) \neq 0\},$$  is an effective divisor whose associated line bundle is $\theta_v(w)\boxtimes \theta_w(v)$ \cite{LeP}. Thus, it defines a map $$SD: H^0(M_H(v), \theta_v(w)) \to H^0(M_H(w), \theta_w(v))^*,$$ which in some cases turn out to be an isomorphism. 
\end{enumerate}

A well-known example of strange duality isomorphism is between a pair of rank one vectors, that is, the moduli spaces are Hilbert schemes of points. In \cite{MO}, A. Marian and D. Oprea interpret O'Grady's maps $\Psi_r: M_H(v_r)\dashrightarrow X^{[a]}$ as induced by Fourier-Mukai transforms, and use that to propagate Hilbert schemes strange duality to a large family of pairs of vectors. Here is one of their theorems which we will base on.

\begin{thm}[\cite{MO}, Theorem 2]\label{SD}
Let $X$ be the elliptic K3 surface as in \cref{sec1.1} with a polarization $H:=c+mf,$ $m>>0.$ Give $v, w$ be Mukai vecotrs with $(v, w^\vee) =0$ and of ranks $r$ and $s$ respectively, satisfying the following:
\begin{enumerate}[(i).]
    \item $r, s\geq 2;$

\item $c_1(v).f=1,$ $c_1(w).f=1;$

\item $(v,v)+(w,w)\geq 2(r+s)^2.$
\end{enumerate}

\noindent Then the duality map $SD: H^0(M_H(v),\theta_v(w))\to H^0(M_H(w), \theta_w(v))^*$ is an isomorphism.
\end{thm}

With our wall-crossing interpretation of O'Grady's maps $\psi_r: M_H(v_r)\dashrightarrow M_H(v_{r+1}),$ we obtain the following based on \cref{SD}. 

\begin{Ex}\label{ex1}
 Let $X$ and $H$ be as in \cref{SD}, suppose that $v, w$ is a pair Mukai vectors with $(v, w^\vee)=0.$ Denote $r, s$ the rank of $v, w$ respectively. If $v$ and $w$ in addition satisfy:
 \begin{enumerate}[(i).]
     \item $r, s \geq 0,$ $r+s\geq 4.$
     
     \item $c_1(v).f=1,$ $c_1(w).f=1;$

\item $(v,v)+(w,w)=2(r+s)^2.$

 \end{enumerate}
 Then, strange duality holds for the pair $v$ and $w.$
\end{Ex}

\noindent Note that condition $(iii)$ is much more restrictive, compare to that of \cref{SD}. The only gain in this example is that one of the vectors can have rank $0$ or $1$ while the other have a higher rank.

To state the second example, we need to generalize the setting of strange duality. In \cite{BM2}, Bayer and Macr\`i extend the definition of determinant line bundles (\cref{detlb}) and therefore strange duality to Moduli of complexes. Let $X$ be any K3 surface, $v, w$ be Mukai vectors with $(v, w)=0$ and $\sigma \in Stab(X)$ be a stability condition. They consider a strange duality morphism $$SD_\sigma: H^0(M_\sigma(v), \theta_v^\sigma(w)) \to H^0(M_\sigma(w), \theta_w^\sigma(v))^*,$$ defined by the effective divisor $\Theta_\sigma:=\{(E, F)\in M_\sigma(v)\times M_\sigma(w)| Ext^1(F, E)\neq 0\}.$

\begin{rmk}
To recover the strange duality morphism $SD$ considered by Marian and Oprea, $\sigma$ should be chosen in the Gieseker chamber of $v$ but not in that of $w.$ Indeed, we should have $M_H(v)= M_\sigma(v)$ and $M_H(w)\cong M_\sigma(-w^\vee).$
\end{rmk}

One can ask how does the morphism $SD_\sigma$ change as the stability condition $\sigma$ varies. In the case when $(w,w)=0, (v,v)\geq 2$, they prove that crossing a totally semistable wall of $v$ "annihilates" $SD_\sigma.$ More precisely, let $\mathcal W$ be a totally semistable wall of $v,$ $\sigma_\pm$ be two stability conditions that are separated by $\mathcal W$ and sufficiently close. Suppose that $SD_{\sigma_+}$ is an isomorphism, then $SD_{\sigma_-} = 0$ (see \cite{BM2}, proposition 15.1). We observe that this can be generalized as follow: 

\begin{prop}\label{prop4}
Let $X$ be any $K3$ surface, $v, w$ be orthogonal Mukai vectors with $(v, v)>0, (w, w)>0$. $\mathcal W$ denotes a nonisotrpic totally semistable wall of $w$ (resp. $v$) and $\sigma_\pm$ two stability conditions separated by $\mathcal W.$ Suppose that $SD_{\sigma_+}$ is well-defined and isomorphic, then $SD_{\sigma_-}$ is also defined and moreover: 
\begin{enumerate}
    \item if $(v, w_0)\neq 0$ (resp. $(w, v_0)\neq 0$), then $SD_{\sigma_-}=0$.
    \item otherwise, $SD_{\sigma_-}$ is an isomorphism.
\end{enumerate}
\end{prop}

\noindent Here $w_0$ is a Mukai vector appearing in the wall-hitting transformations, uniquely determined by $w$ and the wall $\mathcal W$ (see \cref{BM6.6}).

\begin{Ex}\label{ex3}
Let $X$ be the elliptic K3 with a polarization $H:=c+mf,$ $m>>0.$ Suppose $v, w$ is a pair of Mukai vectors with $(v, w)=0.$ Denote $r, s$ the rank of $v, w$ respectively. Suppose in addition that: 
\begin{enumerate}[(i).]
    \item $r, s \geq 3;$
    \item $c_1(v).f=1,$ $c_1(w).f=1;$

\item $(v,v)+(w,w)=2(r+s)(r+s-2).$
\end{enumerate}
Then, there is a stability condition $\sigma'$ such that $$SD_{\sigma'}: H^0(M_\sigma(v), \theta_v^{\sigma'}(w)) \to H^0(M_\sigma(w), \theta_w^{\sigma'}(v))^*$$ is isomorphic. As we will see, this $\sigma'$ is separated from the place where $SD_\sigma = SD$ by a totally semistable wall, and consequently $SD: H^0(M_H(v), \theta_v(w)) \to H^0(M_H(w), \theta_w(v))^*$ is the zero morphism, by \cref{prop4}. 
\end{Ex}

Note that in this example, the quantity $(v, v)+(w,w)$ goes beyond the lower bound in condition $(iii)$ of \cref{SD}. The price we pay is a wall-crossing.

The rest of this paper is organized as following: in \cref{sec2} we review some preliminaries, in \cref{sec3} we prove \cref{thm0} and in \cref{sec4} we justify the remarks about strange duality.

\subsection*{Acknowledgement.} I would like to thank Aaron Bertram for his advices that make this preprint possible, and for his constant support and encouragement. I am also indebted to Emanuele Macr\`i who introduced me to the problem and taught me about the subject. This research was partially supported by NSF FRG grant DMS-1663813.

\section{Review on stability conditions and wall-crossing for moduli.} \label{sec2} 
We collect here some preliminaries, and introduce Bayer and Macr\`i's work on wall-crossing for moduli of complexes on K3 surfaces \cite{BM2}. Throughout this section, $X$ is any K3 surface.

\subsection{Mukai lattice of $X$.} 
\begin{dfn}
The Mukai lattice of $X$ is $H^*_{alg}(X):=H^0(X)\oplus H^{1,1}(X, \mathbb{Z})\oplus H^2(X)$ with the Mukai pairing $((v_0, v_2, v_4), (w_0, w_2, w_4)):= v_2 w_2 - v_0 w_4 - w_0 v_0.$ We call an element in this lattice a Mukai vector.
\end{dfn}

\begin{dfn}
Given an object $E\in D^b(X)$ in the derived category of coherent sheaves on $X,$ define its Mukai vector $v(E):=ch(E)\sqrt{td(X)}\in H^*_{alg}(X),$ where $ch(E)$ is the Chern character of $E,$ and $td(X)=(1, 0, 2)$ is the Todd class of $X.$
\end{dfn}

\noindent Note that this defines an additive map $v(-): D^b(X) \to H^*_{alg}(X).$ 

\begin{prop}[e.g. \cite{Huy1}, Chapter 5]
\begin{enumerate}[(a)]
    \item Given $E, F\in D^b(X),$ then $$-\chi(E, F):=\sum_i (-1)^{(i-1)}ext^i(E, F) = (v(E), v(F)).$$ 

\item Suppose that $\Phi_P : D^b(X)\to D^b(Y)$ is a Fourier-Mukai transformation with kernel $P\in D^b(X\times Y),$ and $\Phi_{v(P)}: H^*_{alg}(X) \to H^*_{alg}(Y)$ is the correspondence, then the following diagram commutes:

{\centering
\begin{tikzcd}
D^b(X) \arrow[d, "v(-)"] \arrow[r, "\Phi_P"] & D^b(Y) \arrow[d, "v(-)"] \\
H^*_{alg}(X) \arrow[r, "\Phi_{v(P)}"] & H^*_{alg}(Y)
\end{tikzcd}
\par}

\end{enumerate}
\end{prop}

\subsection{Bridgeland stability conditions and moduli of complexes.}
\begin{dfn}
A Bridgeland stability condition on $D^b(X)$ is a pair $(Z, \mathcal A),$ where $Z: H^*_{alg}(X)\rightarrow \mathbb C$ is a group homomorphism and $\mathcal A$ is a heart of a bounded t-structure of $D^b(X),$ satisfying\ all\  the\  following\  conditions:
\begin{enumerate}[(1)]
\item For any nonzero object $E \in \mathcal A,$ $Z(E):=Z(v(E))= r(E)e^{i\pi\phi(E)}, $ then $r(E)>0$ and $\phi(E)\in (0,1].$ ($\phi(E)$ is called the phase of $E$, and it defines a notion of semistability: an object $E\in \mathcal A$ is semistable if for any nonzero subobject $F\hookrightarrow E$, $\phi(F)\leq \phi(E).$)
\item For any object $E\in \mathcal A,$ E has a Harder-Narasimhan filtration $$0\hookrightarrow E_1 \hookrightarrow \ldots \hookrightarrow E_{n-1} \hookrightarrow E_n= E ,$$ such that quotient objects $A_i\cong E_i/E_{i-1}$ are semistable with decreasing phases, i.e. $\phi(A_i)>\phi(A_{i+1})$ for $i=1, 2, \ldots n.$

\item For a given norm $\| \cdot \|$ on $H^*_{alg}(X)$, there exists a constant real number $C>0$ such that $$|Z(E)|<C\|v(E)\|$$ for every semistable object in $E\in \mathcal A.$
\end{enumerate}
\end{dfn}

\begin{dfn}
An object $E\in D^b(X)$ is $\sigma$-semistable if $E[i]\in \mathcal A$ is semistable.
\end{dfn}

\begin{Ex}[\cite{Bri2}, Proposition 7.1]\label{ex2.1}
Given $\beta, \omega \in NS(X)\otimes \mathbb R$ with $\omega$ ample, T. Bridgeland shows that a stability condition $\sigma_{\beta, \omega}$ can be constructed as follow: define $$Z_{\beta, \omega}:=(exp(\beta+i\omega), -),$$ where (- , -) is the Mukai paring, and $exp(\beta+i\omega):=(1, \beta, \frac{\beta^2-\omega^2}{2}) + i(0, \omega, \beta\omega).$
Also define two additive subcategories of the abelian category of coherent sheaves  $\mathcal T:=\{E\in Coh(X):\mu_{\omega}(E/E_{tor})\geq \beta \omega \},$ and $\mathcal F:=\{E\in Coh(X)\ torsion\ free:\mu_{\omega}(E)< \beta \omega \},$ where $\mu_\omega(E):=\frac{c_1(E).\omega}{r(E)}.$ Then the tilt category $\mathcal A_{\beta, \omega}$ with respect to $\mathcal T, \mathcal F$, namely, $$\mathcal A_{\beta, \omega}:=\{E\in D^b(X): \mathcal H^{-1}(E)\in \mathcal F, \mathcal H^{0}(E)\in \mathcal T, \mathcal H^{i\neq -1, 0}(E)\cong 0\}$$ is a heart of a bounded t-structure of $D^b(X),$ and moreover, the pair $(Z_{\beta, \omega}, \mathcal A_{\beta,\omega})$ is Bridgeland stability condition.
\end{Ex}

\noindent Let $Stab(X)$ denote the set of all Bridgeland stability conditions on $D^b(X).$

\begin{thm}[\cite{Bri1}, Theorem 1.2]
$Stab(X)$ admits a complex manifold structure of dimension 2+$\rho(X)$, where $\rho(X)$ is the Picard rank of $X.$  
\end{thm}

\noindent Let $Stab^\dag(X)$ denote the connected component of $Stab(X)$ containing $\sigma_{\beta, \omega}$ (defined in \cref{ex2.1}). The following result is well-known.

\begin{prop}
Given a Mukai vector $v,$ there exists a locally finite set of walls (real codimension 1 subsets) of $stab^\dag(X),$ such that
\begin{enumerate}
\item when $\sigma$ varies within a chamber, the set of $\sigma-$semistable objects  does not change;
\item when $\sigma$ cross a wall, the set of $\sigma-$semistable objects changes.
\end{enumerate}
\end{prop}
\noindent See e.g. \cite[Proposition 2.3]{BM1}.

\begin{dfn}\label{slice}
Given an ample divisor $H,$ we get a two dimensional slice of $Stab^\dag(X)$ defined as $P_H:=\{\sigma_{uH, tH}: u\in \mathbb R, t>0\}.$
\end{dfn} 

The wall-and-chamber structure on $P_H$ is particularly neat:
\begin{lemma}[\cite{ABCH}, Section 6]
Given a Mukai vector $v,$ then set of walls for $v$ intersects $P_H$ at either semicircles or vertical rays, and they do not intersect each other.
\end{lemma}

Fix $v$ as above, we say a stability condition $\sigma$ is generic with respect to $v$ if it does not lie on a wall of $v$. The existence of moduli of $\sigma$-semistable objects of class $v$ as a scheme was proved in \cite{BM1}, based on a result of Y. Toda \cite{Toda}.
\begin{thm}[\cite{BM1}, Theorem 1.3]
Assume that $\sigma$ is generic with respect to $v.$ The set of $\sigma-$semistable objects in the heart $\mathcal A_{\sigma}$ of class $v$ forms a (coarse) moduli space, denoted by $M_{\sigma}(v).$ It is an irreducible normal projective variety. Moreover, there is an open subset $M_{\sigma}^{st}(v)$ (possibly empty) of $M_{\sigma}(v)$ which parametrizes $\sigma-$stable objects; when $v$ is primitive in $H^*_{alg}(X),$ $M_{\sigma}(v)=M_{\sigma}^{st}(v)$ is smooth and projective.  
\end{thm}

These moduli of complexes should be viewed as variants of moduli of (twisted) Gieseker stable torsion-free sheaves, according to the following proposition by T. Bridgeland.

\begin{prop}[\cite{Bri2}, Proposition 14.2]\label{prop2.2.6}
Given $\beta, \omega \in NS(X)\otimes \mathbb R$ with $\omega$ ample, and $E\in D^b(X)$ with positive rank and the imaginary part $Im(Z(E))>0$, then $E$ is of $\sigma_{\beta, t\omega}$-semistable for all $t$ sufficiently large if and only if $E$ is the shift of a $(\beta, \omega)-$Gieseker semistable torsion-free sheaf.
\end{prop}

Such a chamber of $v$ where semistable objects are (shifts of) Gieseker semistable sheaves is called the Gieseker chamber of $v$.

\subsection{Wall-crossing.}\label{sec2.3}

Now we review some of the work in \cite{BM2} about wall-crossing for $M_\sigma(v).$  Throughout this subsection, $v$ is primitive with $(v, v)>0.$ When $\mathcal W$ is a wall of $v,$  $\sigma_0$ denotes a generic element in $\mathcal W$ (i.e. it does not lie in any other wall), $\sigma_+$, $\sigma_-$ denote generic stability conditions separated by $\mathcal W,$ and sufficiently closed to $\sigma_0.$ 

\begin{thm}[\cite{BM2}, Theorem 1.1]\label{BMthm1}
$M_{\sigma_+}(v)$ and $M_{\sigma_-}(v)$ are birational. Moreover, there is a derived equivalence $\Phi : D^b(X) \to D^b(X)$ induces a birational map $\phi : M_{\sigma_+}(v) \dashrightarrow M_{\sigma_-}(v)$ in the sense that: there exists a big open subset $U \subset M_{\sigma_+}(v)$ such that $\phi$ is an isomorphism from $U$ to its image, and for any $u\in U,$ $\Phi(\mathcal{E}_{u})= \mathcal F_{\phi(u)},$ where $\mathcal{E}_u$ and $\mathcal{F}_{\phi(u)}$ are the semistable objects that parametrized by the points $u$ and $\phi(u)$ respectively.  
\end{thm}

The main technical tool behind this theorem is their classification of walls. The following sublattice of $H^*_{alg}(X)$ plays a key role: $$\mathcal H_{\mathcal W}:=\{u\in H^*_{alg}(X): Z_\sigma(u)\in \mathbb R Z_\sigma(v), \forall \sigma \in \mathcal W  \}.$$ 

\begin{prop}[\cite{BM2}, Proposition 5.1]
$\mathcal H_{\mathcal W}$ is a primitive hyperbolic sublattice of rank two containing $v$. If $E$ is a $\sigma_+$-stable object of class $v$, then the classes of its $\sigma_0$-Jordan-H\"older factors and its $\sigma_-$-Harder-Narasimhand factors are contained in $\mathcal H_{\mathcal W}.$
\end{prop}

Remarkably, $\mathcal H_{\mathcal W}$ contains enough information to determine the wall-crossing behavior of $M_\sigma(v)$ at $\mathcal W.$ To give a precise statement, we need some definitions.

\begin{dfn}
Given an arbitrary primitive hyperbolic rank two sublattice $\mathcal H,$ define a potential wall $\mathcal W$ associated to ${\mathcal H}$ as a connected component of the codimension one submanifold $\{\sigma \in Stab^\dag(X): Z_\sigma(\mathcal H) \subset \mathbb R e^{i\pi\phi}\}.$
\end{dfn}

\begin{dfn}
Let $\mathcal W$ be a potential wall associated to $\mathcal{H},$ define the effective cone of $\mathcal W$ $$C_{\mathcal W}:=\{u\in \mathcal H\otimes \mathbb R: (u,u)\geq -2, Z_\sigma(u)\in \mathcal R_{>0}Z_\sigma(v), \forall \sigma\in \mathcal W\}.$$ 
\end{dfn}

The following lemma justifies the name of $C_{\mathcal W}:$

\begin{lemma}[\cite{BM2}, Proposition 5.5]
If $u\in C_{\mathcal W} \cap \mathcal H,$ then for every $\sigma\in \mathcal W$ there exist a $\sigma$-semistable object of class $u.$ If $u\notin C_{\mathcal W},$ then for a generic $\sigma \in \mathcal W,$ there does not exist any $\sigma$-semistable object of class $u.$
\end{lemma}

\begin{dfn}
A class $u\in \mathcal H$ is called: effective if $u\in C_{\mathcal W};$ isotropic if $(u, u)=0;$ spherical if $(u, u)=-2.$ 
\end{dfn}

\begin{prop}[\cite{BM2}, Proposition 5.7]\label{BM5.7}
Given $v$, $\mathcal H$ as above and $\mathcal W$ a potential wall associated to $\mathcal H,$ then $\mathcal W$ is a totally semistable wall of $v$ if and only if in $\mathcal H$ there exists either an isotropic class $w$ with $(v,w)=1,$ or an effective spherical class $s$ with $(v, s)<0.$
\end{prop}
\begin{rmk}
Their proposition also clarifies when $\mathcal W$ induce a divisorial contraction, a samll contraction, when it is a fake wall (being totally semistable but inducing an isomorphism between moduli) and when it is not a wall.
\end{rmk}

We will need to further classify totally semistable walls, using \cref{BM5.7} and the following.

\begin{prop}[\cite{BM2}, Lemma 8.3]\label{BM8.3} Suppose that $\mathcal W$ is a totally semistable wall for $v$ such that $\mathcal H_{\mathcal W}$ contains an isotropic class, then there exists a unique $\sigma_0$-stable spherical object $S$ with $v(S)\in \mathcal C_{\mathcal W}.$ Furthermore, if $E\in M_{\sigma_+}(v)$ generic in moduli, then its HN filtration with at $\sigma_-$ has length two and of the form either $S^{\oplus a}\to E \to F$ or $F\to E\to S^{\oplus a},$ where $a>0.$
\end{prop}

\begin{cor}\label{clastsw}
If $\mathcal W$ is a totally semistable wall of $v,$ and $E\in M_{\sigma_+}(v)$ generic in moduli, then there exists either a spherical subobject or spherical quotient object $S$ of $E$ with $(E,S)<0$, or the class decomposes as $v=s+n w',$ where $s$ spherical, $w'$ primitive and isotropic, and $n$ is the number such that $(v,v)=2n-2.$
\end{cor}

\begin{proof}
Suppose that we have an effective $(s,v)<0.$ Let $S\in M_{\sigma_0(s)},$ then $(s,v)=ext^1(S, E)-hom(S,E)-ext^2(S,E)<0$ implies that we have either $hom(S, E)\neq 0$ or $hom(E, S)\neq 0.$ Since $E$ is $\sigma_0$-semistable, $S$ is either a subobject or quotient of $E.$ Suppose otherwise, then we have $(v, w)=1$ where $w$ is isotropic by \cref{BM5.7}. So \cref{BM8.3} applies and $v=as+v':=av(S)+v(F).$ Furthermore, the proof of \cite[proposition 8.4]{BM2} shows that we must have $a = 1\ or\ n-1,$ $(v', v')=0,$ $(s,v')=n$ and $v'=nw',$ where $n$ is the number such that $(v,v)=2n-2$ and $w'$ is a primitive isotropic class. We observe that if $a=n-1,$ then $(v,s)=(n-1)(s,s)+(v',s)=2-n.$ Since we assume that $(v,s)\geq 0,$ so $n=2.$ So in any case $a=1.$ Thus $v=s+nw'.$  
\end{proof}

\begin{dfn}
We refer to the former case in \cref{clastsw} as a spherical totally semistable wall, and the latter as a Hilbert-Chow totally semistable wall. Also, we call a total semistable wall $\mathcal W$ isotropic if $\mathcal H_{\mathcal W}$ contains an isotropic class, and nonisotropic otherwise.
\end{dfn} 

As we shall see, Hilbert-Chow totally semistable walls are relatively simple for our purpose: they are some vertical rays on a slice $P_H \subset Stab(X).$ On the other hand, spherical totally semistable walls will be the main character in \cref{sec3}: we need to sort out the first wall among these. Once that is done, we will see the first wall is in fact nonisotropic. So we will need some of Bayer and Macr\`i's analysis on nonisotropic totally semistable walls [\cite{BM2}, Section 6]. First, recall an important lemma by Mukai:
\begin{lemma}[\cite{Muk}, Corollary 2.8]
Suppose that $\mathcal A$ is a heart of a bounded t-structure of $D^b(X),$ and $0\to A\to E\to B\to 0$ is a short exact sequence in $\mathcal A,$ such that $Hom(A, B)=0.$ Then $ext^1(E, E)\geq ext^1(A, A)+ext^1(B,B).$ 
\end{lemma}

\begin{cor}[c.f. \cite{HMS}, Section 2]\label{mukaicor}
If $S$ is a $\sigma$-semistable spherical object, then all its Jordan-H\"older factors are spherical ($\sigma$-stable object).
\end{cor}

\begin{prop}[\cite{BM2}, proposition 6.3]
Suppose that $\mathcal W$ is a nonisotropic totally semistable wall, and $\sigma_0 \in \mathcal W$ generic. Then $\mathcal H_{\mathcal W}$ contains infinitely many spherical classes, and two of them, denoted by $s$ and $t,$ admits a stable object in their moduli, that is $M_{\sigma_0}(s)=M_{\sigma_0}^{st}(s)=\{S\}$ and $M_{\sigma_0}(t)=M_{\sigma_0}^{st}(t)=\{T\}.$ Other spherical $\sigma_0$-semistable objects are strictly semistable, whose Jordan-H\"older factors are necessarily $S$ or $T$ (by \cref{mukaicor}).
\end{prop}

Consider $G \subset Aut(\mathcal H)$ the subgroup generated by spherical twists $\{\rho_s: s\in \mathcal H\ $effective and spherical$\}.$

\begin{prop}[\cite{BM2}, proposition and definition 6.6]\label{BM6.6}
Given $v\in \mathcal H$ with $(v,v)>0$ and effective, the $G$-orbit of $v$ contains a unique class $v_0$ satisfying $(v_0,s)\geq 0$ for any effective spherical class $s\in \mathcal H.$
\end{prop}

\begin{thm}[\cite{BM2}, proposition 6.8]\label{BMthm2} Let $\mathcal W$ be a nonisotropic totally semistale wall, $\sigma_o, \sigma_+ , \sigma_-$ and $v_o$ as before. Then one has the following diagram: 

\begin{center}
\begin{tikzcd}[column sep=tiny]
M_{\sigma_+}(v) & & M_{\sigma_-}(v) \\
       & M_{\sigma_o}(v_o)  \arrow[ul, dashed, "\phi_+"] \arrow[ru, dashed, "\phi_-"]
\end{tikzcd} 
\end{center} 

\noindent Here the morphisms $\phi_\pm$ are isomorphic in codimension 1 and induced by derived equivalences $\Phi_\pm$ respectively  (in the sense of \cref{BMthm1}), where $\Phi_+$ and $\Phi_-$ are either a composition of finitely many spherical twists or inverse spherical twists, e.g. $$\Phi_+ = ST^{-1}_{\mathcal S^+_1}\circ \dots \circ ST^{-1}_{\mathcal S^+_n}, \ \ \ \ \Phi_- = ST_{\mathcal S^-_1}\circ \dots \circ ST_{\mathcal S_n^-},$$ where $\mathcal S_i^\pm$ are $\sigma_\pm$-stable spherical objects, and $ST_{\mathcal S^\pm_i}$ denotes the spherical twist with respect to $\mathcal S_i^\pm.$  
\end{thm}

\begin{dfn}\label{intermediate}
We call $\phi_{\pm}$ (resp. $\Phi_{\pm}$) in the above wall-hitting transformations (resp. wall-hitting derived equivalences) at a sphercial totally semistable wall $\mathcal{W}.$ 
\end{dfn}

\subsection{Line bundles on $M_\sigma(v)$.} We recall the definition of determinant line bundles in various setting:

\begin{dfn}\label{detlb}
Suppose $v, w$ are primitive Mukai vectors with $v^2\geq 0$ and $(v,w^\vee)=0.$
\begin{enumerate}[(1)]
\item On fine moduli of sheaves $M_H(v):$ with a universal family $\mathcal U$, it is defined via the Fourier-Mukai transform with kernel $\mathcal U:$ $$\theta_v(w):= det^{-1}\Phi_{\mathcal U}(F)= det[R{p_Y}_*(\mathcal U\otimes^L Lp_X^*(F))]^{-1},$$  where $F$ is any sheaf on $X$ of class $w.$ 
\item On coarse moduli of sheaves $M_H(v)$: one can still do the same construction on the defining Quot scheme over $M_H(v)$, the resulting line bundle descends because of the condition $(v, w^\vee)=0.$ See for example \cite[Chapter 8]{Huy2}.
\item On moduli of complexes $M_{\sigma}(v)$: given the above assumption of $v,$ and in addition that $\sigma$ is generic with respect to $v$, then there exists a quasi universal family $\mathcal E$ in the sense that, $\mathcal E|_{[E]\times X}\cong E^{\oplus \rho}.$ In this case $\theta_v^\sigma(w) \in NS(M_{\sigma}(v))$ is defined as a numerical divisor class via intersection numbers $$\theta_v^\sigma(w^\vee).C=-\frac{1}{\rho}(w^\vee, \Phi_{\mathcal E}(\mathcal O_C)),$$ where $C$ is a curve in $M_{\sigma}(v).$ See \cite[section 5]{BM1}.
\end{enumerate}
\end{dfn}

\begin{prop}[\cite{BM1}, proposition 4.4]\label{det}
The definitions are compatible, namely, if $M_\sigma(v)=M_H(v)$, then $\theta_v^\sigma(-w^\vee) = \theta_v(w)$.
\end{prop}
\begin{proof}
\begin{align*} 
det^{-1}\Phi_{\mathcal U}(F).C &=-ch_1(\Phi_{\mathcal U}(F)).C \\
              &=-ch_1(\Phi_{\mathcal U}(F)|_C)  \\  
             &=-\chi(C, \Phi_{\mathcal U}(F)|_C), \ as\ rk(\Phi_{\mathcal U}(F))=\chi(v\cdot w^\vee)=0, \\
              &=-\chi(C\times X, q^*F\otimes \mathcal U|_C) \\
              &=-\chi(X, F\otimes \Phi_{\mathcal U}(\mathcal O_C)) \\
              &=(w^\vee, v(\Phi_{\mathcal U}(\mathcal O_C))).
\end{align*}

\end{proof}

%\begin{rmk}
%To realize strange duality morphisms $SD$ considered by Marian and Oprea as the ones $SD_\sigma$ considered by Bayer and Macr\`i, a necessary and sufficient condition is that the line bundles $\theta_v(w)$ and $\theta_w(v)$ should be movable and big, see \cite[proposition 15.1]{BM2}. We remark that if $v, w$ satisfy the conditions in \cref{SD}, then the line bundles are big and nef: indeed, in \cite{MO}, $\theta_v(w)$ and $\theta_w(v)$ are identified with tautological line bundles $L^{[a]}$ and $L^{[b]}$ on Hilbert schemes via O'Grady's maps. Here $L$ is a line bundle on $X.$ The condition $(iii)$ in \cref{SD} implies $L$ is big and nef, thus so do $\theta_v(w)$ and $\theta_w(v).$
%\end{rmk}

\section{O'Grady's birational maps via wall-crossing.}\label{sec3}

In this section, we prove \cref{thm0}. We return to the setting in \cref{sec1.1}, where $X$ is an elliptic K3 surface with $NS(X)\cong \mathbb Zc \oplus \mathbb Zf$.  Also, we fix a polarization $H:=c+mf,$ for $m$ sufficiently large. Such a polarization has the property that being $H$-Gieseker stable is equivalent to being $H$-slope stable, so in particular, twisting by a line bundle preserves $H$-Gieseker stability. 

In O'Grady's construction, we define $ F_1 := \mathcal{I}_Z \otimes \mathcal O_X(c+nf)$ so that  $\chi({F}_1) =1,$ and $ \tilde F_1 := F_1\otimes \mathcal O(-2f)$ with $\chi (\tilde F_1) = -1.$ Let $v_1$ and $\tilde v_1$ denote the Mukai vectors of $F_1$ and $\tilde F_1,$ respectively.

For $ \tilde F_1 \in M_H(\tilde v_1)$ general, one has a unique nontrivial extension, which is also Gieseker $H-$semistable, $$0\rightarrow \mathcal O_X \rightarrow F_2 \rightarrow \tilde F_1 \rightarrow 0.$$ Assume the Euler charateristic of the extension $\chi({F}_2)=1,$ and if we define $\tilde F_2:={F}_2(-2f),$ then $\chi(\tilde F_2)=-1.$ So we repeat the extension process and obtain recursively $\tilde F_{r-1} := F_{r-1}\otimes \mathcal O(-2f)$, and $F_r$ as the unique extension of $\tilde F_{r-1}$ by $\mathcal O_X.$

To put things in Bayer-Macr\`i's setting, we rename $E_1:=F_1=\mathcal I_Z(c+nf)$ and define $E_r$ for $r\geq 0$ via $$ E_r = ST^{-1}_{\mathcal S_{r-1}} (E_{r-1}), $$ where $\mathcal S_{r} := \mathcal O(2rf).$ Also we define Mukai vectors $v_r := v(E_r),$ and $s_r := v(\mathcal S_r).$ Notice that for any $r\geq 1,$ $(v_r, s_{r-1}) = -1 $ and $E_r \otimes \mathcal O(-2(r-1)f) \cong F_r. $ 

Notice that the class $v_r$ here differs from that in \cref{sec1.1} by twisting of a line bundle, but the moduli spaces are isomorphic.

We break our proof of \cref{thm0} into the following two propositions.

\begin{prop}\label{prop1}
There exists a path $\gamma$ crossing $\mathcal W_i$ in decreasing order, such that each $\mathcal W_i$ is the first totally semistable wall of $v_i$ along $\gamma.$
\end{prop}

\begin{prop}\label{prop2}
At each wall $\mathcal W_{i+1},$ one of the wall-hitting derived equivalences is exactly $$ST^{-1}_{\mathcal S_i}(E_i) \rightarrow E_i \rightarrow RHom(\mathcal S_i[*], E_i)\otimes S_i[*].$$ 
\end{prop}

%Here is a sketch of this section. First, in \cref{sec3.1}, we describe the first totally semistable wall for ideal sheaf of $n$ points on a slice $P_{c+nf}$ (\cref{slice}), which also gives the first totally semistable wall of $v_1$ on $P_{c+nf}$. Next, we prove \cref{prop2} in \cref{sec3.2}. Along the way we will see that in order to have $\mathcal W_r$ appear on $P_{c+mf},$ $m$ has to be sufficiently large. Thus, in \cref{sec3.3} we deform the slice from $P_{c+nf}$ to $P_{c+mf},$ $m>>0$ and show that $\mathcal W_1$ is always the first totally semistable wall of $v_1.$ Finally, we use O'Grady's result and \cref{prop2} to prove \cref{prop1}. 

\subsection{The first totally semistable wall for ideal sheaves.}\label{sec3.1}  Let $Z$ denotes a 0-dimensional subscheme of length $n \geq 2$, $I_Z$ its ideal sheaf and $X^{[n]}$ the Hilbert scheme. Write $w:= v(I_Z) = (1, 0, 1-n).$ Given an ample divisor $H$, define $P_H := \{ \sigma_{uH, tH} |\ t> 0, u\in \mathbb R\}.$ The goal of this section is to prove the following.

\begin{prop}\label{prop3}
There exist a path of the form $\gamma_1 :=\{\sigma_{uH, tH} | \ t: \infty \to 0\},$ such that the first totally semistable wall of $w$ along $\gamma_1$ is caused by $\mathcal O(-C),$ where $C$ is a genus $n$ curve.
\end{prop}

\begin{lemma}\label{NoH-C}
Choose $H=c+nf,$ the only Hilbert-Chow totally semistalbe wall of $I_Z$ on $P_H$ is the vertical wall $\{u=0\}.$  
\end{lemma}
\begin{proof}
By \cref{clastsw}, we have $(1,0,1-n)=s+nw'$ where $s$ spherical, $w'$ primitive and isotrpic. Write $w'=(r, c_1, p),$ then using the conditions $(s,s)=-2$ and $(w,w)=0$ we get $c_1^2-2rp=0$ and $1+r+p=nr.$ Define $d_h=c_1.H,$ the equation of this wall on $P_H$ is $\frac{d_h}{2}(u^2+t^2)+(r(1-n)-p)u+(n-1)d_h =0.$ Note that the discriminant $$\Delta= [(n-1)r+p]^2-2d_h^2(n-1) \leq [(n-1)r+p]^2-4rp(n-1)= [(n-1)r-p]^2 = (1+r+p-r-p)^2=1,$$ where equality holds if and only if $c_1=kH,$ for some $k\in \mathbb Z.$ Thus, $\Delta >0$ if and only if $c_1=kH,$ for some $k\in \mathbb Z.$ 

We claim that in fact $c_1=0.$ Indeed, given that $c_1=kH,$ the two conditions above become $k^2(n-1)=rp$ and $1+p+r=nr.$ Note that all the variables here are integers, which is impossible unless $k=0.$ Thus, either $w=(0,0,-1)$ or $n=2, w=(1,0,0).$ In any case, the equation of the wall is $u=0.$

\end{proof}

\begin{proof}[Proof of \cref{prop3}]
By \cref{NoH-C}, we can only encounter spherical totally semistable walls. Recall that this kind of walls is caused by a sphercial quotient or subobject $S$ with $(S, I_Z)<0.$

We first consider a wall $\mathcal W_E$ that caused by a spherical subobject $E \hookrightarrow I_Z$ with $(E, I_Z) \leq 0.$ Choose an ample divisor $H$ with $H^2 =1,$ and $G$ orthogonal to $H$ with $G^2 = -1,$ then we can write $ch(E) = (r, d_hH+d_gG, ch_2),$ where $d_h:=c_1(E).H$ and $d_g:=-c_1(E).G.$ We summarize some useful properties of $E$ in the following lemma and prove it below.

\begin{lemma}\label{lemma1}
$E$ is a torsion-free sheaf, so in particular has rank $r\geq 1.$ If $r=1,$ then $E=\mathcal O(-C),$ where $C$ is an effective curve on $X.$ Moreover, $d_h := c_1(E).H <0.$
\end{lemma}

Consider the 2-dimensional slice of stability conditions $P_H := \{ \sigma_{uH, tH} |\ t> 0, u\in \mathbb R\}.$ The equation of the wall $\mathcal W_E$ on the slice is
\begin{equation}\label{eq1}
d_h(u^2+t^2)-2(nr+ch_2)u+2nd_h = 0.
\end{equation}
It defines a semicircle if the discriminant $\Delta = (nr+ ch_2)^2-2n{d_h}^2 >0.$ Note that we have $c_1(E)^2 = {d_h}^2-{d_g}^2 \leq {d_h}^2$ and $-2 = (E, E) = c_1(E)^2 - 2r(ch_2+r),$ therefore, 
\begin{equation}
\begin{split}
\Delta & = 2n(2rch_2 - {d_h}^2) + (nr - ch_2)^2 \\
  & \leq  2n(2rch_2 - {c_1(E)}^2) + (nr-ch_2)^2 \\
  & =  -4n(r^2-1) + ((I_Z, E)+2r)^2
 \end{split}
\end{equation}
For the last equality, we also used $$(I_Z, E) = ((1, 0, 1-n), (r, c_1, ch_2+r))= nr-ch_2 - 2r.$$

Recall that by assumption $(I_Z, E) < 0.$ On the other hand, we have a lower bound of $(I_Z, E):$

\begin{lemma}\label{lemma2}
Suppose that $E$ is a destablizing subobject of $I_Z$ that causes a wall on the 2-dimensional slice $P_H,$ then $(I_Z, E)>-4r,$ where $r=ch_0(E).$
\end{lemma}
We will prove this lemma below. Given the lower bound, we have $$\Delta \leq -4n(r^2-1)+(2r-1)^2=-4(n-1)r^2 -4r + (4n+1).$$ 

%Since $E$ is a subobject of $I_Z$ in a tilted heart, $E$ is a torsion-free sheaf, and therefore its rank $r\geq 1.$ 

If $r\geq 2,$ then $\Delta \leq -4(n-1)r^2 -4r +(4n+1) < 0$ and $E$ does not define a wall on the 2-dimensional slice $P_H.$

Otherwise $r=1.$ Then the only possibilities to have $\Delta>0$ are $(E, I_Z) = -1 $ or $ -4r+1$. By \cref{lemma1}, $E= \mathcal O(-C).$ Recall that the condition $(E, I_Z) = -1$ or $-4r+1$ is equivalent to $nr-ch_2 = 1$ or $-2r+1$, and consequently $n- ch_2 = n - \frac{1}{2}C^2 = n - (g-1) = \pm 1,$ where $g$ is the genus of $C.$ This shows that the first totally semistable wall is caused by $\mathcal O(-C).$ 

From now on we fix our polarization $H=c+nf.$ Then the first totally semistable wall along a ray $\gamma_1:=\{\sigma_{uH, tH} | t:\infty \to 0\}$ is caused by $\mathcal O(-C) \hookrightarrow I_Z$ with $C = c+nf,$ as long as $u$ is suitable such that $\gamma_1$ cross it. Let $\mathcal W_{\mathcal O(-C)}$ denotes this wall.

Next, we need to consider wall $\mathcal W_Q$ that are caused by a spherical stable quotient object $I_Z \twoheadrightarrow Q$ with $(I_Z, Q)\leq 0.$ By abuse of notation, we set $ch(Q):=(r, d_hH+d_gG, ch_2).$ The following lemma states some useful properties of $Q$:

\begin{lemma}\label{lemma3}
Suppose $Q$ as above is a spherical stable quotient of $I_Z$ with $(I_Z, Q)\leq 0.$ Then $Q \cong F[1]$ is the shifting of a torsion-free sheaf $F$, $d_h > 0$ and $r<0.$
\end{lemma}

 We claim $\mathcal W_Q$ does not prevent $\mathcal W_{\mathcal O(-C)}$ from being the first totally semistable wall. To prove this, note that the equation of $\mathcal W_{\mathcal O(-C)}$ is
 $$(u + \sqrt{2(n-1)}+ \dfrac{1}{\sqrt{2(n-1)}})^2+t^2 = \dfrac{1}{2(n-1)}.$$
 In particular, the right end point of $\mathcal W_{\mathcal O(-C)}$ (which is a semicircle) is $u= -\sqrt{2(n-1)}.$  

On the other hand, for $Q$ to be in the heart $\mathcal A_{uH, tH},$ we should have $\mu_H(Q) = \dfrac{d_h}{r} \leq u.$ 

Now choose a vertical ray $\gamma_u:=\{\sigma_{uH, tH}|t:\infty \to 0\}$ that crosses $\mathcal W_{\mathcal O(-C)}.$ If there is a wall $\mathcal W_Q$ that prevent $\mathcal W_{\mathcal O(-C)}$ from being the first totally semistable wall along $\gamma_u$, then we have $\dfrac{d_h}{r} \leq u \leq -\sqrt{2(n-1)} $. However, this would implies $d_h^2 \geq 2(n-1)r^2,$ and consequently the discriminant of the equation of $\mathcal W_Q$ would be $$\Delta = (nr+ch_2)^2-2nd_h^2 \leq 4r^2(n-1)^2 - 4r^2n(n-1) <0.$$

Therefore, $\mathcal W_{\mathcal O(-C)}$ is the first totally semistable wall along some path $\gamma_1:=\{\sigma_{uH,tH} | t:\infty \to 0\}.$
\end{proof}

\begin{proof}[Proof of \cref{lemma1}]
Let $ 0\rightarrow E \rightarrow I_Z \rightarrow Q \rightarrow 0$ be the short exact sequence in the heart, then we have an exact sequence of coherent sheaves: $$0\rightarrow H^{-1}(E) \rightarrow 0\rightarrow H^{-1}(Q) \rightarrow E \rightarrow I_Z \rightarrow H^0(Q)\rightarrow 0.$$ 
Hence, $E$ is a sheaf. Note that the image of $E\rightarrow I_Z$ must be of the form $I_\Gamma(-C),$ where $\Gamma$ is a 0-dimensional subscheme and $C$ an effective curve on $X$. 

On the other hand, $H^{-1}(Q)\in \mathcal F$ is torsion-free and $\mu_H(H^{-1}(Q))< \mu_H(I_Z) = 0,$ that is $c_1(H^{-1}(Q)).H < 0.$ Thus, being an extension of $I_\Gamma(-C)$ by $H^{-1}(Q),$ $E$ is torsion-free with $d_H:= c_1(E).H <0.$

If we assume in addition that $r=1,$ then $E=I_\Gamma(-C),$ but then $(E, E)=-2$ would imply $\Gamma = \emptyset$ and $E = \mathcal O(-C).$ 
\end{proof}

\begin{proof}[Proof of \cref{lemma2}] We want to show $(E, I_Z) > -4r$, or equivalently $nr-ch_2 > -2r$. Assume for contradiction that we have $nr-ch_2 \leq -2r,$ then $ch_2 \geq (n+2)r,$ so let's write $ch_2 = (n+2+p)r,$ for some rational number $p\geq 0.$ 
Notice that $${d_h}^2 \geq c_1^2 = 2rch_2 + 2(r^2-1) \geq 2(n+2+p)r^2.$$ Therefore, $d_h \leq - r\sqrt{2(n+2+p)},$ since $d_h$ is negative by \cref{lemma1}.

Now we consider the imaginary part of the central charge of $E$ at $\sigma_{uH, tH}:$  Im$ Z(E)= t(d_h-ru) \geq 0,$ as we assume $E$ is in the heart. This implies that $$u\leq \dfrac{d_h}{r} \leq -\sqrt{2(n+2+p)}.$$ On the other hand, the point ($u= -\sqrt{2(n+2+p)}, t= 0$) satisfies the equation \eqref{eq1} of the potential wall caused by $E$ (to check this, note that if $u=-\sqrt{2(n+2+p)},$ then $d_h = ru$), and the center of the semicircle defined by \cref{eq1} is $$\dfrac{rn+ch_2}{d_h} \geq -\dfrac{r(2n+2+p)}{r\sqrt{2(n+2+p)}} > - \sqrt{2(n+2+p)},$$ we see that on the wall $\mathcal W_E$, $E$ is not in the heart. This is a contradiction.
\end{proof}

\begin{proof}[Proof of \cref{lemma3}]
We consider the following distinguished triangle $$H^{-1}(Q)[1]\rightarrow Q \rightarrow H^0(Q).$$ Suppose that $H^0(Q) \neq 0,$ then we have $Q \rightarrow H^0(Q) \rightarrow Q_1 \rightarrow Q_2,$ where $Q_1$ is the most destablizing quotient of $H^0(Q)$ and $Q_2$ is a stable factor of $Q_1.$ Thus, the composition $Q \rightarrow Q_2$ is surjective and $Q_2$ is itself spherical by \cref{mukaicor}. Therefore $Q \cong H^0(Q) \cong Q_2.$ This means we have a short exact sequence $0\rightarrow E=I_\Gamma(-D) \rightarrow I_Z \rightarrow Q= \mathcal O_{\Gamma \cup D}(-Z) \rightarrow 0$ in $Coh (X).$ For $Q$ to be spherical, $D$ must be a $-2$-curve. For $Ext^1(Q, Q)= 0,$ we see  that $\Gamma = \emptyset.$ But then $(I_Z, Q)=1,$ contradict to our assumption.

Thus, $Q\cong F[1]$ for some torsion-free sheaf with $\mu_H(F)<u<0,$ and therefore, $d_h>0.$
\end{proof}

\begin{rmk}\label{rmk2} 
In the case when $n=2,$ we choose $H=\dfrac{c+(2 + \epsilon)f}{\sqrt{2+2\epsilon}}$ for some $0< \epsilon <<1,$ so that $H$ is ample and $\mathcal O(c+2f)$ still cause the first totally semistable wall. 

In any case, our choice of $H$ is suitable to $w = (1, 0, 1-n)$ in sense of O'Grady \cite{O'G1}.
\end{rmk}

\begin{cor}\label{cor1}
$\mathcal W_1$ is the first totally semistable wall of $E_1=I_Z(c+nf)$ along some vertical path on the slice $P_H.$
\end{cor}
\begin{proof}
Choose $H=\dfrac{c+nf}{\sqrt{2n-2}}$ as above, $\mathcal W_{\mathcal O(-H)}$ causes the first totally semistable wall of $I_Z.$ Twisting by $\mathcal O(H)$ induces an action on $Stab(X)$: $(\otimes \mathcal O(H)).\sigma_{uH, tH} = \sigma_{(u-\sqrt{2n-2})H, tH}.$   
\end{proof}

\subsection{The first totally semistable wall of $v_1$.}\label{sec3.3}
 Denote $P_m$ the 2-dimensional slice $\{\sigma_{uH, tH}:t>0\},$ where $H:=\frac{c+mf}{\sqrt{2m-2}}$. While \cref{cor1} states that when $m=n,$ the first totally semistable wall of $E_1 = I_Z(c+nf)$ on $P_n$ is caused by $\mathcal O$, the other walls $\mathcal W_r$ may not appear on $P_n.$ Indeed, \cref{lemma4} below shows that we need to let $m$ be sufficiently large, in order to have the walls $\mathcal W_r,$ $1\leq r \leq R,$ all appear on $P_m$. 
 
 \begin{lemma}\label{lemma4}
Given $R\geq 1,$ if we choose the polarization $H=\frac{\mathcal O(c+mf)}{\sqrt{2m-2}}$ with $m$ sufficiently large, then the 2-dimensional slice $P_H$ intersects $\mathcal W_r,$ for $1\leq r \leq R$. Moreover, on this slice, the walls $\{\mathcal W_r: 1\leq r\leq R\}$ are nested semicircles: $\mathcal W_{r-1}$ is inside $\mathcal W_r$ on $P_H.$
\end{lemma}

\begin{proof}
The wall $\mathcal W_r$ of $v_r$ caused by $s_{r-1}$ is described as: 
\begin{equation}\label{eq7}
(r^2-r+2-n-m)(u^2+t^2)-2\sqrt{2m-2} u + 4(r-1)=0.
\end{equation}
Our assumption guarantees that this semicircle has positive radius and its center is on the $\{u\leq 0\}$ half.

From the equations we can see that $\mathcal W_{r-1}$ and $\mathcal W_r$ have no intersection. On the other hand, they both intersect $t$-axis, at $t_{r-1}=2\sqrt{\frac{r-2}{n+m+3r-4-r^2}}$ and $t_r=2\sqrt{\frac{r-1}{n+m+r-2-r^2}}$ respectively. Note that $t_{r-1} < t_r$ Thus, $\mathcal W_{r-1}$ is contained in $\mathcal W_r.$
\end{proof}

Thus, we deform $P_m$ by varying $m$ from $n$ to $\infty$, to show that $\mathcal W_1$ is always the first totally semistable wall for $v_1$ on $P_m,$ for $m\geq n.$

\begin{prop}\label{lemma6}
For all $m\geq n,$ $\mathcal W_1$ is the first totally semistable wall of $v_1$ on $P_{m},$ along a vertical ray starting from the Gieseker chamber of $v_1$. 
\end{prop}

\begin{proof}
Assume for contradiction that on $P_{m_1},$ for some $m_1>n,$ $\mathcal W_1$ is not the first totally semistable wall along any vertical ray. This means that on $P_{m_1}$ there exists a bigger totally semistable wall $\mathcal W'$ of $v_1$ (semicircle) contains $\mathcal W_1,$ because walls of a fixed Mukai vector are nested. Then, $\mathcal W_1$ and $\mathcal W'$ must coincide on some $P_{m_0},$ $m_0>n$. Suppose $\mathcal W'$ is caused by another spherical object $\mathcal S'.$

Note that at $\mathcal W_1 \cap P_{n},$ the Jordan-H\"older filtration is $$0\to \mathcal O \to I_Z(C) \to \mathcal O_C(-Z)\otimes \mathcal O(C)\to 0,$$ where $C:=c+nf.$ By \cref{lemma5}, $\mathcal O$ is always stable near $\mathcal W' \cap P_{m_0}.$ So $\mathcal S'$ must be contained in $\mathcal O_C(-Z) \otimes \mathcal O(C).$ Since $(\mathcal S', \mathcal O_C(-Z))\otimes \mathcal O(C))<(\mathcal S', I_Z(C)<0$, $\mathcal S'$ also causes a totally semistable wall of $v_0=v(\mathcal O_C(-Z) \otimes \mathcal O(C)).$ 

Thus we have a destablizing sequence $E'\to \mathcal O_C(-Z)\otimes \mathcal O(C) \to Q'$ with either $E'$ or $Q'$ being the stable spherical object $\mathcal S'$. Now consider $$0\to H^{-1}(Q') \to E' \to \mathcal O_C(-Z)\otimes \mathcal O(C) \to H^0(Q') \to 0,$$ here $H^0(Q')$ is supported on points at most, since we can assume $\mathcal O_C(-Z)\otimes \mathcal O(C)$ is generic in moduli and therefore $C$ is irreducible. Also $H^{-1}(Q')$ is not zero, because otherwise $Q'$ and $E'$ cannot possibly be spherical.

Write $ch(E') = (r, d_hH+d_gG, ch_2),$ where $H=\frac{c+mf}{\sqrt{2m-2}}$ and $G=\frac{c+(2-m)f}{\sqrt{2m-2}}.$ Then the equation of $\mathcal W'$ on $P_H$ is $$(2-m-n)r(u^2+t^2)-2r\sqrt{2m-2}u + 2((m+n-2)ch_2+\sqrt{2m-2}d_h)=0.$$ Compare it to the equation of $\mathcal W_1$ (see \cref{eq7}), one sees that a necessary and sufficient condition for two walls to overlap is 
\begin{equation}\label{eq8}
(m+n-2)ch_2+\sqrt{2m-2}d_h=0.
\end{equation}
For each $(u,t)\in \mathcal W'\cap P_H,$ we have $\frac{d_h(E')}{r(E')}\geq u.$ Note that $r(E')>0$ because $H^{-1}(Q')\neq 0.$ Since $u$ can be sufficiently close to $0,$ we see that $d_h(E')\geq 0.$ On the other hand, $d_h(H^{-1}(Q'))<ur(H^{-1}(Q'))<0$ and therefore $$0\leq d_h(E')<d_h(\mathcal O_C(-Z)\otimes \mathcal O(C))=\dfrac{m+n-2}{\sqrt{2m-2}}.$$ By \cref{eq8}, we then have $0\leq ch_2(E') <1.$ Because $X$ is a K3 surface, $ch_2(E')$ is integral. Thus $ch_2(E')=0$ and $d_h(E')=0.$ 

Consequently, the Mukai vector $v(E')=(r, k(c+(2-m)f), r)$ and $(E', E')=-2k^2(m-1)-2r^2,$ so it cannot be spherical unless it is $\mathcal O_X$, but that would contradict our assumption. Thus, $Q'$ has to be spherical. Note that by Mukai's lemma, $Ext^1(H^0(Q'), H^0(Q'))=0$ and therefore $H^0(Q')=0.$ The equality $(Q', Q')=-2$ gives
\begin{equation}\label{eq9}
-(m-1)k^2+n+k(m-n)=r(r+1).
\end{equation}
Given that $r\geq 1$ and $m\geq n,$ \cref{eq9} holds only when $k=0$ and $n=r(r+1).$ In particular, $E'={\mathcal O}^r.$ 

Therefore we have a surjective map $\mathcal O_X^r \rightarrow \mathcal O_C(-Z)\otimes \mathcal O(C),$ which must factor through $\mathcal O_C^r \to \mathcal O_C(-Z)\otimes \mathcal O(C)$. Via the Abel-Jacobi map $$Sym^n(C) \to Pic^{-n}(C)\cong Pic^{n-2}(C),$$ we see that $\mathcal O_C(-Z)\otimes \mathcal O(C)$ is generic in $Pic^{n-2},$ provided that $Z\subset C$ is generic. However, a generic element in $Pic^{n-2}(C)$ is not effective since $Sym^{n-2}(C)\to Pic^{n-2}(C)$ is not surjective. This is a contradiction.
\end{proof}

\begin{rmk}
$\mathcal W_1$ is actually a wall for $\mathcal O_C(-Z)\otimes \mathcal O(C),$ although not totally semistable. Indeed, the destablized objects are precisely those in the Brill-Noether variety $W_{n-2}^{r-1}(C),$ namely, degree $n-2$ line bundles on $C$ that are globally generated with $r$ sections. Note that its dimension $dim\ W_{n-2}^{r-1}(C) = n - r(r+1)=0.$ And the destablizing sequence is $$\mathcal O^r \to \mathcal O_C(-Z)\otimes \mathcal O(C)\to F'[1],$$ where $F'$ is the Lazasfeld-Mukai bundle. See \cite{Bayer}.
\end{rmk}

\subsection{Proofs of \cref{prop1} \& \cref{prop2}}\label{sec3.2} In this subsection we first prove \cref{prop2}, and then use it together with \cref{lemma6} and O'Grady's result to prove \cref{prop1}. Recall $\mathcal W_r, \mathcal S_r, E_r, s_r $ and $v_r$ from the beginning of this section.

\begin{lemma}\label{lemma5}
Choose a polarization $H=\frac{\mathcal O(c+mf)}{\sqrt{2m-2}}$ with $m$ sufficiently large, then the spherical object $\mathcal S_{r-1}$ is stable at the wall $\mathcal W_r$ on $P_H$.
\end{lemma}
\begin{proof}
By a result of Arcara and Miles \cite[theorem 1.1]{AM}, $\mathcal S_{r-1}$ can only be possibly destablized by $\mathcal S_{r-1}(-c),$ where $c$ is the section of the elliptic fibration. Such a wall has the following equation: 
\begin{equation}\label{eq6}
(2-m)(u^2+t^2) + 2\sqrt{2m-2}(2r-1)u - 4(2r-1)(r-1) = 0.
\end{equation}
In particular, it lies on the $\{u\geq 0\}$ half. 

An elementary computation shows that this semicircle \eqref{eq6} has no intersection with that of \eqref{eq7}, and since the latter centers at the $\{u \leq 0\}$ half, these two walls do not contain each other. As $\mathcal S_{r-1}$ is stable at large volume limit, it is stable at the wall $\mathcal W_r$ on $P_H.$
\end{proof}

\begin{proof}[Proof of \cref{prop2}]
We want to prove that at the wall $\mathcal W_r$ of $v_r$ caused by $s_{r-1}$, the wall-hitting derived equivalences $\Phi_\pm$ consist of one spherical twist (or inverse) respectively. Based on Bayer-Macr\`i's analysis, this is amount to show that $v_{r-1}$ pairs with all effective spherical classes non-negatively, in the rank two hyperbolic sublattice $\mathcal H_r$ of the wall. 

Note that by \cref{lemma5}, $\mathcal S_{r-1}$ is stable at the wall. If $\mathcal H_r$ is isotropic, then by \cref{BM8.3}, $s_{r-1}$ is the unique effective spherical class in $\mathcal H_r$ and $(s_{r-1}, v_{r-1})=1.$ If $\mathcal H_r$ is non-isotropic, $s_{r-1}$ is one of the two spherical classes that have a stable object in their moduli. Suppose that $t_0$ is the other one, then $(t_0, s_{r-1}) \geq 3.$ Now let's use $\{v_r, s_{r-1}\}$ as a (rational) basis of $\mathcal H_r$ and write $t=xs_{r-1}+yv_r.$ All spherical classes $t$ lie on a hyperbola: 
\begin{equation}\label{eq3}
-2=(t,t)= -2x^2+(2n-2)y^2-xy. 
\end{equation}
$(t, v_{r-1})=0$ defines a line: 
\begin{equation}\label{eq4}
y = -\dfrac{1}{2n-1}x.     
\end{equation}
Also we have the constraint $(t, s_{r-1})\geq 3,$ for all $t$ that are effective spherical and does not lie on the same branch with $s_{r-1}.$ This gives 
\begin{equation}\label{eq5}
2x+y \leq -3.
\end{equation}
 
\begin{figure}\label{graphs}
    \centering
    \includegraphics[width=10cm]{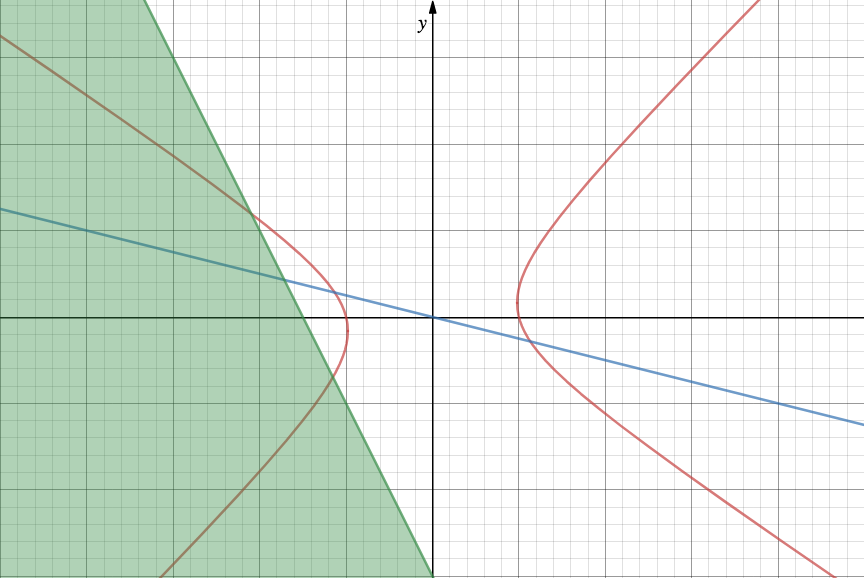}
    \caption{Graphs of \cref{eq3},\cref{eq4} and \cref{eq5}.}
    \label{graphs}
\end{figure}

The graphs of these equations (see \cref{graphs}) show that all effective spherical classes pair with $v_0$ positively. Indeed, on the right branch of the hyperbola, effective spherical classes are all above $s_{r-1}$ and therefore above the line $(v_{r-1},t)=0$. On the left branch, effective classes are to the left of the line $(s_{r-1},t) =3$ and on the upper half-plane, and consequently also above the line $(v_{r-1},t)=0$. 

Thus, $v_{r-1}$ is the minimal class in the orbit of $v_r$ and $\Phi_\pm$ both consist of one spherical twist.
\end{proof}

%With \cref{prop3} and \cref{prop2}, we can now prove \cref{prop1}.

\begin{proof}[Proof of \cref{prop1}] 
We shall let $m$ be sufficiently large, then by \cref{lemma4} we have a path $\gamma:=\{\sigma_{uH, tH} | t:\infty \to 0\}$ for $H=c+mf$ and some suitable $u$ such that it passes through $\mathcal W_i,$ $1\leq i \leq r$ in decreasing order and at generic points of the walls. Let $\gamma(t_i)$ denotes the point when $\gamma$ meets $\mathcal W_i.$

By \cref{lemma6}, $\mathcal W_1$ is the first totally semistable wall for $v_1.$ Thus a generic stable object in the moduli $M_{\gamma(t)}(v_1),$ for $t>t_1,$ is of the form $I_Z(c+nf),$ since it is stable at large volume limit.

Then according to O'Grady's result and \cref{prop2}, the wall-hitting derived equivalence at $\mathcal W_i$ is the unique extension $$0\to \mathcal S_{i-1} \to E_i \to E_{i-1} \to 0.$$ In particular, a generic $E_i$ is stable for $t_i<t<\infty.$ That is, $\mathcal W_i$ is the first totally semistable wall of $v_i$ along $\gamma,$ for all $i\geq 1.$  
\end{proof}

\section{Remarks on strange duality for elliptic K3.}\label{sec4} In \cref{sec4.1}, we prove \cref{prop4}. In \cref{sec4.2}, we use wall-crossing and \cref{SD} to give two examples of strange duality for an elliptic K3 surface.

\subsection{Wall-crossing behavior of $SD_\sigma$.}\label{sec4.1} Throughout this subsection, $X$ is any K3 surface, $v, w$ are orthogonal Mukai vectors. Let $\mathcal W$ be a nonisotropic totally semistable wall of $v$ but not of $w$, $\sigma_0\in \mathcal W$ be generic at the wall, $\sigma_+$ and $\sigma_-$ are separated by $\mathcal W$ and sufficiently close to $\sigma_0$. To simplify notation, denote $SD_\pm$ the maps $SD_{\sigma_\pm}: H^0(M_{\sigma_\pm}(v), \theta_v(w)) \to H^0(M_{\sigma_\pm}(w), \theta_w(v))^*$ described in \cref{sec1.4}. 

\begin{lemma}\label{lemma4.1}
If $SD_+$ is defined, then so is $SD_-.$
\end{lemma}

\begin{proof}
First we note that $\phi_0(v)\neq \phi_0(w),$ where $\phi_0$ is the phase function at $\sigma_0$: suppose otherwise, then $v, w\in \mathcal H_{\mathcal W}$, the rank 2 hyperbolic sublattice associate to $\mathcal W,$ however the assumptions $(v, w)=0$, $v^2>0$, $w^2>0$ imply that $v$ and $w$ are linearly independent but does not span a hyperbolic lattice, contradition.  

Now $SD_+$ is defined, in particular the locus $\Theta_+$ is nonempty, i.e. $Hom(E, F)\neq 0$ for some $(E, F)\in M_{\sigma_+}(v)\times M_{\sigma_+}(w)$, thus we should have $\phi^+(w)>\phi^+(v).$ Then $\phi_-(w)\approx \phi_+(w) > \phi_+(v) \approx \phi^-(v)$ as $\sigma_+$, $\sigma_-$ are sufficiently close. Thus $Ext^2(E', F')\cong Hom(F', E')^*= 0$ for all $(E', F')\in M_{\sigma_+}(v)\times M_{\sigma_+}(w)$. Then by \cite[proposition 9]{LeP}, $SD_-$ are well-defined. 
\end{proof}

Recall that at the totally semistable wall $\mathcal W,$ if $\phi^+(v)>\phi^+(v_0)$, then $\psi_+:M_{\sigma_0}(v_0) \dashrightarrow M_{\sigma_+}(v)$ is induced by a series of spherical twists. Let $E_0\in M_{\sigma_0}^s(v_0)$, $E_{i+1}:=ST_{S_{i+1}^+}(E_i)$ as in \cref{sec2}.

\begin{lemma}\label{lemma4.2}
For $i\geq 0$,  then we have an surjection $Ext^1(ST_{S_{i+1}}(E_i), F)\twoheadrightarrow Ext^1(E_i, F)$, for any $F\in M_{\sigma}(w)$.
\end{lemma}

\begin{proof} To simplify notation, let $E:=E_i$, $S:=S_{i+1}^+$, so that $E_{i+1}= ST_S (E)$. By definition of spherical twist, the defining distinguished triangles $$\bigoplus_{j} Hom(S[j], E)\otimes S[j] \to E \to ST_S (E)$$ gives an exact sequence: $$ Ext^1 (ST_S(E) , F) \to Ext^1 (E, F)\to Ext^{1}(\bigoplus_j Hom(S[j], E])\otimes S[j], F).$$ According to [\cite{BM2}, Proposition 6.8], $E$ and $S$ lie in a heart of $D^b(X),$ by Serre duality $$Ext^{1}(\bigoplus_j Hom(S[j], E])\otimes S[j], F) = \bigoplus_{j=0}^{2} Hom(S, E[j])\otimes Hom(S, F[j+1]).$$

\noindent Note that $Hom(S, F[2])= Hom(F, S)=0$, because $\phi^{\pm}(S)\approx \phi^{\pm}(v) < \phi^{\pm}(w)$. By the "Induction Claim" in the proof of \cite[Proposition 6.8]{BM2}, $S$ and $E$ are both simple in a certain additive subcategory, and therefore $Hom(S, E)= Hom(E, S)= 0$. Hence, $ Ext^{1}(\bigoplus_j Hom(S[j], E])\otimes S[j], F)=0 $. 

\end{proof}

Still assume that $\phi^+(v)>\phi^+(v_0)$, the other map $\psi_-:M_{\sigma_0}(v_0) \dashrightarrow M_{\sigma_-}(v)$ is induced by a series of inverse spherical twists. Let $E_0\in M_{\sigma_0}^{st}(v_0)$ as before, but reset $E_{i-1}:=ST_{S_{i}^-}^{-1}(E_{i})$.

\begin{lemma} \label{lemma4.3}
For $i\leq 0$, we have an injection $Hom(E_i, F)\hookrightarrow Hom(ST_S^{-1}(E_i), F)$, for any $F\in M_{\sigma}(w)$.
\end{lemma}

\begin{proof}
This is dual to \cref{lemma4.2}.
\end{proof}

\begin{lemma}\label{lemma4.4}
 Suppose that $(w, v_0)\neq 0.$ If $SD_+\neq 0$ (resp. $SD_-\neq 0$), then $SD_- =0$ (resp. $SD_+=0$).
\end{lemma}
 
\begin{proof}
Because $Hom(E_+, F)\neq 0$ for some $E_+\in M_{\sigma_+}(v)$, and both $E_+$ and $F$ are $\sigma_0$-semistable, so $\phi_0(v)<\phi_0(w)$. And therefore we have $\phi_0(v_0)= \phi_0(v)<\phi_0(w)$, which implies $Ext^2(E_0, F)=0$ for all $E_0\in M_{\sigma_0}(v_0)$ and all $F\in M_{\sigma_+}(w)$.   

Without loss of generality, we may assume $\phi_+(v)>\phi_+(v_0)$. Thus $\Phi_+:M_{\sigma_0}(v_0) \dashrightarrow M_{\sigma_+}(v)$ is induced by a series of spherical twists. If $0<(v_0, w)= ext^1(v_0, w)- hom(v_0, w)$, then $Ext^1(E_0, F)\neq 0$ for all $(E_0,F)\in M_{\sigma_0}(v_0)\times M_{\sigma_+}(w)$, and by \cref{lemma4.2}, this implies $Ext^1(E^+, F)\neq 0$ for all $(E_+, F)\in M_{\sigma_+}(v)\times M_{\sigma_+}(w)$, which contradicts to our assumption that $SD_+\neq 0.$ Hence, $(v_0, w)<0$ and $Hom(E_0, F)\neq 0$. By \cref{lemma4.3}, $SD_-$ is a zero map.
\end{proof}

Now we turn to the case when $(v_0, w)=0.$ In this case, the rank two hyperbolic lattice $\mathcal H_{\mathcal W}$ is perpendicular to $w.$ Recall that $v=v_l=\sum_{i=1}^l v_0+(v_{i-1}, s_i)s_i$, where $(v_{i-1}, s_i)=ext^1(v_{i-1}, s_i)>0$. Thus, $$\theta_w(v)\cong \theta_w(v_0) \otimes [\otimes_{i=1}^l {\theta_w(s_i)}^{(v_{i-1}, s_i)}] .$$

Denote $\Theta_\pm \subset M_{\sigma_\pm}(v)\times M_\sigma(w)$ the theta divisors, and define $\Theta_0:=\{(E_0, F) \in M_{\sigma_0}(v_0)\times M_\sigma(w): Ext^1(E_0, F)\neq 0\}.$ Then 

\begin{equation}\label{eq1}
\psi_+^*(\theta_v(w)\boxtimes \theta_w(v))\cong [\theta_{v_0}(w)\boxtimes \theta_{w}(v_0)]\otimes q^*(\otimes_i \theta_w(s_i)^{(v_{i-1},s_i)}),
\end{equation}
where $q: M_{\sigma_0}(v_0)\times M_{\sigma}(w) \rightarrow M_{\sigma}(w)$ is the projection, and $\bar s := \sum_{i=1}^l (v_{i-1}, s_i)s_i$. Correspondingly, 
$$\psi_+^{-1}(\Theta_+)= \Theta_0 + \sum_{i=1}^l (v_{i-1}, s_i) q^{-1} D_i^+,$$
where $D_i^+ := \{(S_i^+, F)\in M_{\sigma_+}(s_1)\times M_{\sigma}(w)\cong M_{\sigma}(w): Hom(S_1^+, F)\neq 0\}$ is a divisor with associated line bundle $\theta_w(s_i)$.

Similarly, $\psi_-^{-1}(\Theta_-)= \Theta_0 + \sum_{i=1}^l (v_{i-1}, s_i) q^{-1} D_i^-,$ where $D_i^- := \{(S_i^-, F)\in M_{\sigma_-}(s_1)\times M_{\sigma}(w)\cong M_{\sigma}(w): Hom(S_1^-, F)\neq 0\}.$

\begin{lemma}\label{lemma4.5}
Assume that $\mathcal W$ is a totally semistable wall of $v$ with $(v_0, w)= 0$, then $\psi_+^{-1}(\Theta_+)=\psi_-^{-1}(\Theta_-)$. Consequently, $SD_+=SD_-$.   
\end{lemma}
\begin{proof}
It suffices to show $D_i^+ = D_i^-$, which is amount to show $Hom(S_i^+, F)\neq 0$ if and only if $Hom(S_i^-, F)\neq 0$, for any $F\in M_{\sigma}(w)$. 

$S_i^+$ is $\sigma_+$-stable, hence $\sigma_0$-semistable. If $S_i^+$ is $\sigma_0$-stable then $S_i^+ = S_i^-$ and the claim follows. So we can assume $S_i^+$ is strictly semistable. By Mukai's lemma, all its $\sigma_0-$stable factors are also spherical. According to [\cite{BM2}, Proposition 6.3], there are exactly two $\sigma_0-$stable spherical objects have the same phase with $s_i,$ denoted by $T_1$ and $T_2$. Both of them should appear as stable factors of $S_i^+$, because otherwise $s_i$ would be a multiple of a spherical class, contradicting the assumption that itself is spherical. For the same reason, $S_i^-$ should also contain both stable objects. 

Now suppose $Hom(S_i^+, F)\neq 0$, then either $Hom(T_1, F)\neq 0$ or $Hom(T_2, F)\neq 0$. As $T_1,\  T_2\in _{\mathcal W}\subseteq w^\perp$ and $Ext^2(T_i, F)=0$, $hom(T_i, F)=ext^1(T_i, F),$ for $i=1, 2,$ so that fact that either $Hom(T_1, F)\neq 0$ or $Hom(T_2, F)\neq 0$ implies $Hom(S_i^-, F)\neq 0$. The other direction is also true for the same reason. Thus, $D_i^+ = D_i^-.$
\end{proof}

\begin{proof}[Proof of \cref{prop4}] By \cref{lemma4.1}, \cref{lemma4.4} and \cref{lemma4.5}.

\end{proof}

\subsection{Strange duality via wall-hitting.}\label{sec4.2} Let $v_r,$ $w_s$ be Mukai vectors of ranks $r,$ $s$ respectively and satisfy $(v_r, w_s^\vee)=0.$ Write $v_r:=(r, c+(a+rp)f, p),$ $w_s:=(s, c+(b+rp)f, q),$ so that $(v_r,v_r)=2a-2,$ $(w_s,w_s)=2b-2.$ Note that the assumption $(v_r, w^\vee_s)=0$ becomes $a+b-2=-(r+s)(p+q).$ And the condition $(iii)$ in \cref{SD} is equivalent to $p+q+s+r\leq 0.$ Recall that the strange duality morphism $$SD_{(r,s)}: H^0(M_{H}(v_{r}), \theta(w_{s})) \to H^0(M_{H}(w_{s}), \theta(v_{r}))^*$$ is defined by the theta locus $\Theta_{(r,s)}:=\{(E_r,F_s)\in M_H(v_r)\times M_H(w_s)\ |\ Ext^1(F_s^\vee[1], E_r) \neq 0\}.$

In \cref{sec3.2}, we see that the first totally semsitable walls $\mathcal V_r$ of $v_r, 1\leq r \leq R,$ are nested on a slice $P_m:=\{\sigma_{uH, tH}:t>0\}$ where $H=c+mf,$ for some $m>>0.$ Now we define $\mathcal W_s$ to be the potential wall for $v(F_s^\vee[1])$ caused by the class $\mathcal S_{s-1}^\vee[1],$ where $F_s\in M_H(w_s)$ and $\mathcal S_{s-1}$ is a spherical object that causes the first totally semistable wall for $w_s$.   

\begin{lemma} $\mathcal W_s$ is the first totally semistable wall for $-w_s^\vee$ along some suitable rays on $P_H.$
Given $r, s\geq 0,$ $\mathcal V_r$ intersects $\mathcal W_s$ on the slice $P_m.$
\end{lemma}
\begin{proof}
First, note that $\mathcal W_s$ is symmetric to the first totally wall for $w_s$ caused by $\mathcal S_{s-1},$ about the vertical ray ${u=0}$ on $P_m.$ The first claim then follows from \cite[proposition 2.11]{BM2}.
Recall that the equation of $\mathcal V_r$ on $P_m$ is (\cref{eq7})
\begin{equation*}\label{eq8}
(r^2-r+2-a-m)(u^2+t^2)-2\sqrt{2m-2} u + 4(r-1)=0.
\end{equation*}

Similarly, the equation of $\mathcal W_s$ is
\begin{equation*}\label{eq9}
(s^2-s+2-b-m)(u^2+t^2)+2\sqrt{2m-2} u + 4(s-1)=0.
\end{equation*}
For $m$ sufficiently large, these two equations have a common solution.
\end{proof}

Now choose a sufficiently large $m,$ then on $P_m,$ $\mathcal V_r$ and $\mathcal W_s$ cut out four chambers near their intersection point. Denote $\sigma_0$ their intersection point.

\begin{lemma}
Suppose that both $\theta_v(w)$ and $\theta_w(v)$ are movable line bundles. Choose a stability condition $\sigma$ above $\mathcal V_r$ and $\mathcal W_s$ on $P_m$, and sufficiently close to $\sigma_0.$ Then the strange duality morphism  $$SD^{\sigma}_{(r, s)}: H^0(M_{\sigma}(v_{r}), \theta^\sigma(-w_{s}^\vee)) \to H^0(M_{\sigma}(-w_{s}^\vee), \theta^\sigma(v_{r}))^*$$ is the same as $SD_{(r,s)}.$
\end{lemma}
\begin{proof}
There exists a large voluem limit $\sigma_\infty$ such that $M_{\sigma_\infty}(v_r) = M_H(v_r)$ by \cref{prop2.2.6}, and moreover $\theta^{\sigma_\infty}(-w_{s}^\vee)=\theta(w_s)$ by \cref{det}. 

Also, by \cite[proposition 2.11]{BM2}, we can choose the large volume limit $\sigma_\infty$ such that $M_{\sigma_\infty}(-w_s^\vee)\cong M_H(w_s),$ induced by the derived equivalence $(-)^\vee[1].$ So $\theta^{\sigma_\infty}(v_{r})$ is identified with $\theta(v_r)$ under the isomorphism. Also, the theta divisors $\Theta^{\sigma_\infty}_{(r,s)}$ and $\Theta_{(r,s)}$ also get identified. Therefore, $SD^{\sigma_\infty}_{(r,s)}=SD_{(r,s)}.$

Since $\sigma$ is above both $\mathcal V_r$ and $\mathcal W_s,$ there eixsts a path $\gamma'$ in $Stab^\dag(X),$ going from $\sigma_\infty$ to $\sigma$ and crossing no divisorial nor totally semistable walls. Indeed, this is an argument in the proof of \cite[proposition 15.1]{BM2}: consider the chamber structure on $Stab^\dag(X)$ cut out by all totally semistable walls for $v_r$ and $-w_s^\vee$. Note that \cref{thm0} shows that $\sigma$ and $\sigma_\infty$ lie in a common chamber $\mathcal C_{tot}$ . Also, under the map $l_v: U(X) \to NS(M_\sigma(v_r))$ (resp. $l_w: Stab^\dag(X) \to NS(M_\sigma(-w_s^\vee))$) defined in \cite[theorem 10.2]{BM2}, where $U(X)\subset Stab^\dag(X)$ is the geometric chamber, $\sigma_\infty$ is mapped to an interior point of the movable cone of $M_\sigma(v_r)$ (resp. $M_\sigma(-w_s^\vee)$), thus there exist an open subset $V$ containing $\sigma, \sigma_\infty,$ contained in both $U(X)$ and the totally semistable chamber $\mathcal C_{tot},$ and maps to the interior of the movable cones of $M_\sigma(v_r)$ and $M_\sigma(-w_s^\vee)$ under $l_v$ and $l_w$ repsectively. Choose a path in this open subset $V,$ then it corsses no totally semistable nor divisorial walls, because divisorial walls of $v$ (resp. $w$) are send to the boundary of the image of $l_v$ (resp. $l_w$) (\cite[lemma 10.1]{BM2}). Thus, $SD^{\sigma_\infty}_{(r,s)} = SD^\sigma_{(r,s)}.$

\end{proof}

\begin{rmk}
Under the assumption of \cref{SD}, the line bundles $\theta_v(w)$ and $\theta_w(v)$ are movable. Indeed, via O'Grady's birational maps, they are identified with tautological line bundles $L^{[a]}$ and $L^{[b]}$ on the Hilbert schemes $X^{[a]}$ and $X^{[b]}$ respectively, where $L=\mathcal O_X(-(p+q+r+s)f),$ which is nef by condition (iii) in \cref{SD}. Thus, the tautological line bundles are nef. Since O'Grady's birational maps are  isomorphic in codimension one, $\theta_v(w)$ and $\theta_w(v)$ are movable.
\end{rmk}

Now we can compare the strange duality morphisms. To facilitate computation, we "normalize" our sheaves. Given $E_r$ of class $v_r=(r, c+(a+rp)f, p),$ define $\tilde{E_r}:=E_r\otimes \mathcal O((-p-r+1)f)$ such that $\chi(\tilde{E_r})=1.$ Then $\mathcal V_r$ is a totally semistable wall of $\tilde{E_r}$ caused by $\mathcal O.$ Note that $H^0(E_r\otimes F_s) = H^0(\tilde{E_r}\otimes \tilde{F_s}((p+q+r+s-2)f)).$

\begin{lemma}Suppose that $\tilde{E_r}$ is generic in moduli, then
\begin{enumerate}[(a).]
    \item $ST_{\mathcal O}(\tilde{E_r}) = \tilde{E_{r-1}}(-2f).$
     
        \item $ST_{\mathcal O(2f)}^{-1}(\tilde{E_r}) = \tilde{E_{r+1}}(2f).$

\end{enumerate}
\end{lemma}
\begin{proof}
These are precisely O'Grady's construction. 
\end{proof}

\begin{prop}\label{propex}
\begin{enumerate}[(a).]
    \item If $p+q+r+s=0$, then $SD_{(r,s)}=SD_{(r+1, s-1)}.$

    \item If $p+q+r+s=-2$, then $SD_{(r,s)}=SD_{(r+1, s+1)}^{\sigma'} ,$ where $\sigma'$ is the strange duality morphism defined at some $\sigma'\in Stab(X).$ Consequently, assume in addition that $SD_{(r,s)}$ is an isomorphism, then $SD_{(r+1, s+1)}=0.$
\end{enumerate}
\noindent Combine with \cref{SD}, part (a) yields \cref{ex1} and part (b) gives \cref{ex3}.
\end{prop}

\begin{proof}[Proof of \cref{propex}]
For (a), if $p+q+r+s=0,$ then 
\begin{equation*}
    \begin{split}
 Hom(F_s^\vee, E_r) & = Hom(\tilde{F_s}^\vee, \tilde{E_r}(-2f)) \\
 & = Hom(ST_{\mathcal O}^{-1}(\tilde{F_s}^\vee), ST_{\mathcal O}^{-1}(\tilde{E_r}(-2f))) \\
 & = Hom(ST_{\mathcal O}(\tilde{F_s})^\vee, \tilde{E_{r+1}}) \\ 
 & = Hom((\tilde{F_{s-1}}, \tilde{E_{r+1}}(-2f))\\
 & = Hom(F_{s-1}^\vee, E_{r+1}).
    \end{split}
\end{equation*}
Thus, $SD_{(r,s)}=SD_{(r+1,s-1)}.$ For (b), given $p+q+r+s=-2,$ then 
\begin{equation*}
    \begin{split}
 Hom(F_s^\vee, E_r) & = Hom(\tilde{F_s}^\vee, \tilde{E_r}(-4f)) \\
 & = Hom(ST_{\mathcal O}((\tilde{F_s}(-2f))^\vee), ST_{\mathcal O}(\tilde{E_r}(-2f))) \\
 & = Hom(ST_{\mathcal O}^{-1}(\tilde{F_s}(-2f))^\vee, ST_{\mathcal O(2f)}(\tilde{E_r})(-2f))) \\ 
 & = Hom(\tilde{F_{s+1}}^\vee, ST_{\mathcal O(2f)}(\tilde{E_r})(-2f))).
    \end{split}
\end{equation*}
As $ST^{-1}_{\mathcal O(2f)}$ inducing an wall-hitting transformation at the totally semistable wall $\mathcal V_r$, $ST_{\mathcal O(2f)}$ induces the other wall-hitting transformation to the other side of the wall. Thus $E':= ST_{\mathcal O(2f)}(\tilde{E_r})(-2f))$ is $\sigma'$-stable, for some $\sigma'$ on the other side of $\mathcal V_r$. The equalities shows that $$SD_{(r,s)}=SD^{\sigma'}_{(r+1,s+1)}.$$ Note that this wall is caused by $\mathcal O(2f)$ and $(\tilde{F_{s+1}}^\vee, \mathcal O(2f))=-3\neq 0,$ thus by \cref{prop4}, $SD_{(r+1, s+1)}=0.$
\end{proof}

\bibliographystyle{alpha}
\bibliography{references}

\end{document}